\documentclass[11pt]{article}
\usepackage{amsmath,amssymb,amsthm}
\usepackage[a4paper,margin=1in]{geometry}
\usepackage{booktabs}
\usepackage{enumitem}
\usepackage{float}
\usepackage{tikz}
\usetikzlibrary{angles,quotes}
\usepackage[hidelinks]{hyperref}

\newtheorem{theorem}{Theorem}[section]
\newtheorem{proposition}[theorem]{Proposition}
\newtheorem{lemma}[theorem]{Lemma}
\newtheorem{corollary}[theorem]{Corollary}
\newtheorem*{theoremA}{Theorem A}
\newtheorem*{theoremB}{Theorem B}
\theoremstyle{definition}
\newtheorem{definition}[theorem]{Definition}
\newtheorem{example}[theorem]{Example}
\newtheorem{question}[theorem]{Question}
\theoremstyle{remark}
\newtheorem{remark}[theorem]{Remark}

\numberwithin{equation}{section}

\newcommand{\R}{\mathbb{R}}
\newcommand{\Sph}{\mathbb{S}^2}
\newcommand{\ip}[2]{\langle #1,#2\rangle}
\newcommand{\cQ}{\mathcal{Q}}
\newcommand{\eps}{\varepsilon}
\newcommand{\sgn}{\operatorname{sign}}
\newcommand{\diam}{\operatorname{diam}}
\newcommand{\defeq}{\mathrel{\mathop:}=}
\newcommand{\eqdef}{=\mathrel{\mathop:}}

\title{Napoleonic triangles on the sphere from similar chord triangles}
\author{Quang Hung Tran}
\date{}

\begin{document}
\maketitle

\begin{abstract}
We study a variant of Napoleon's construction on the unit sphere, in which the triangles erected on the sides of a spherical triangle $ABC$ are required to have similar isosceles chord triangles with apex angle $\beta\in(0,\pi/2)$. For a non-equilateral $ABC$, we prove that the triangle $XYZ$ of spherical circumcentres is equilateral if and only if $\cos(\theta_a+\theta_b+\theta_c\pm\Omega/2)=0$, where $\theta_a,\theta_b,\theta_c$ are the base angles of the isosceles spherical triangles $XBC$, $YCA$, $ZAB$, $\Omega$ is the area of $ABC$, and the sign $+$ (resp.~$-$) corresponds to the outward (resp.~inward) construction. We derive from this an explicit classification in terms of two quadrics and a cubic inequality. It recovers the theorem of Dipierro, Noakes and Valdinoci for $\beta=\pi/3$, and exhibits three families of solutions for $\beta>\pi/3$, together with a fourth one when $\cos\beta<1/4$.
\end{abstract}

\noindent\textit{Keywords:} Napoleon's theorem, spherical triangle, spherical excess, chord triangle.\\
\textit{2020 Mathematics Subject Classification:} 51M09, 51M04.

\section{Introduction}\label{sec:intro}

\subsection{Main results}
Let $ABC$ be a triangle in the Euclidean plane. Napoleon's theorem asserts that the centres of the equilateral triangles erected on the sides of $ABC$, all outwards or all inwards, form an equilateral triangle (see e.g.~\cite[\S3.3]{CG67}); in the inward construction this triangle degenerates to a point when $ABC$ is equilateral. On the sphere the corresponding statement fails in general, as observed numerically by Treuden \cite{Tre95}. Dipierro, Noakes and Valdinoci \cite{DNV24} classified the spherical triangles for which it holds. Besides the equilateral triangles, there is a two-parameter family of outward solutions, parametrized by an ellipsoid, whose triangles of centres are all congruent \cite[Theorem~1.7]{DNV24}; the inward centres have equal pairwise distances only for an equilateral input, in which case they coincide. In the hyperbolic plane, only equilateral triangles have the Napoleon property \cite{DNV25}; we also refer to McKay \cite{McK01} for an earlier hyperbolic analogue of a different nature.

In this paper we allow the erected triangles to vary in shape. In the plane, the centre of an equilateral triangle on $BC$ sees $BC$ under the angle $2\pi/3$, whereas its opposite vertex sees $BC$ under the angle $\pi/3$. Erecting similar isosceles triangles with an arbitrary apex angle leads to Kiepert's configuration (cf.~\cite{EF94}), and a short computation with complex numbers shows that, for a non-equilateral input, their apices form an equilateral triangle only when the apex angle is $2\pi/3$, i.e.~in Napoleon's case. On a sphere of fixed radius similar triangles are congruent, so that this generalization is not available as such. We therefore impose the similarity condition on the \emph{chord triangles}, i.e.~on the Euclidean triangles in $\R^3$ spanned by points of the sphere.

More precisely, we fix once and for all an angle $\beta\in(0,\pi/2)$. On the side $BC$ of a spherical triangle $ABC$ we choose a point $D$ of the sphere such that the chord triangle $DBC$ is isosceles with Euclidean angle $\beta$ at $D$, and we let $X$ be the spherical circumcentre of $D,B,C$. The sides $CA$ and $AB$ give rise to $E,Y$ and $F,Z$ in the same way, the points $D,E,F$ being chosen all outwards or all inwards. This is possible as soon as the sides of $ABC$ are shorter than $2\beta$ (see \S\ref{sec:construction}). When $\beta=\pi/3$ the chord triangles are equilateral, and we recover the construction of \cite{DNV24}.

It turns out (Lemma~\ref{lem:construction}) that $XBC$ is an isosceles spherical triangle with apex angle $2\beta$ and acute base angles. We stress that $\beta$ is the chord angle at $D$, whereas $2\beta$ is the spherical angle at the circumcentre $X$. We write $\theta_a$ for the base angle of $XBC$, define $\theta_b,\theta_c$ cyclically, and set $\Theta\defeq\theta_a+\theta_b+\theta_c$. We denote by $\Omega$ the area of $ABC$. Our first main result is the following criterion.

\begin{theoremA}[Theorem~\ref{thm:angle}]
Let $ABC$ be a non-equilateral spherical triangle with sides shorter than $2\beta$. The triangle $XYZ$ of the outward construction is equilateral if and only if $\cos(\Theta+\Omega/2)=0$; for the inward construction the condition is $\cos(\Theta-\Omega/2)=0$. In this case $XYZ$ is nondegenerate, and its side $\ell$ satisfies $\cos\ell=\sigma\cos\beta/(1-\sigma\cos\beta)$, where $\sigma=\pm1$ is the sine of $\Theta\pm\Omega/2$.
\end{theoremA}

In particular, for a fixed $\beta$, the non-equilateral inputs satisfying the criterion produce equilateral triangles $XYZ$ of at most two sizes. We invite the reader to compare Theorem~A with the planar situation: there $\Omega=0$ and every base angle equals $\pi/2-\beta$, so that the criterion reduces to $\cos(3\pi/2-3\beta)=0$, i.e.~to $\beta=\pi/3$. On the sphere, both $\Theta$ and the area term depend on the input triangle. In Remark~\ref{rem:limit} we show that, for small triangles and $\beta=\pi/3$, the leading discrepancy in the inward construction is precisely the defect in Weitzenb\"ock's inequality.

The criterion of Theorem~A has an algebraic counterpart. We set $x=\cot\theta_a$, $y=\cot\theta_b$, $z=\cot\theta_c$, and for $\kappa\in\{\cos\beta,-\cos\beta\}$ we let $\cQ_\kappa$ be the quadric
\begin{equation}\label{eq:intro-quadric}
\kappa^2(x^2+y^2+z^2)+2\kappa^3(xy+yz+zx)=(1-\kappa)^2(1+2\kappa).
\end{equation}

\begin{theoremB}[Theorem~\ref{thm:quadric}, Corollaries~\ref{cor:small} and~\ref{cor:large}]
Let $ABC$ be as in Theorem~A, and let $\eps=1$ for the outward and $\eps=-1$ for the inward construction.
\begin{enumerate}[label=(\arabic*)]
\item The construction with sign $\eps$ gives an equilateral $XYZ$ if and only if, for some $\kappa\in\{\cos\beta,-\cos\beta\}$, we have $(x,y,z)\in\cQ_\kappa$ and $\eps\kappa(x+y+z-xyz)>0$. Conversely, every point of $\cQ_\kappa$ in $(0,\tan\beta)^3$, outside the diagonal $x=y=z$ and the surface $x+y+z=xyz$, determines a spherical triangle whose construction with sign $\eps=\sgn(\kappa(x+y+z-xyz))$ gives an equilateral $XYZ$.
\item If $\beta\le\pi/3$, non-equilateral solutions only occur in the outward construction. They form a two-parameter family, parametrized by the part of the ellipsoid $\cQ_{\cos\beta}$ in the open positive octant with its diagonal point removed. For $\beta=\pi/3$ this is the family of \cite{DNV24}.
\item If $\beta>\pi/3$, there are two outward families and one inward family, and a second inward family exists if and only if $\cos\beta<1/4$.
\end{enumerate}
\end{theoremB}

The situation is summarized in Table~\ref{tab:summary}. Here a \emph{family} refers to the direction of the construction and to the size of $XYZ$. We point out that the corresponding sets of input triangles need not be disjoint: a non-equilateral triangle may give an equilateral $XYZ$ in both constructions (Remark~\ref{rem:both}). Furthermore, whenever the chosen construction gives an equilateral $XYZ$ with side $\ell>\pi/2$, the corresponding triangle $DEF$ is equilateral as well, while this need not be the case when $\ell<\pi/2$ (Proposition~\ref{prop:DEF} and Remark~\ref{rem:DEF}).

\begin{table}[ht]
\centering
\renewcommand{\arraystretch}{1.25}
\begin{tabular}{@{}lccc@{}}
\toprule
 & $0<\beta\le\pi/3$ & $\pi/3<\beta\le\arccos\frac14$ & $\arccos\frac14<\beta<\pi/2$\\
\midrule
outward, $\cos\ell=-\frac{\cos\beta}{1+\cos\beta}$ & yes & yes & yes\\
outward, $\cos\ell=\frac{\cos\beta}{1-\cos\beta}$ & no & yes & yes\\
inward, $\cos\ell=\frac{\cos\beta}{1-\cos\beta}$ & no & yes & yes\\
inward, $\cos\ell=-\frac{\cos\beta}{1+\cos\beta}$ & no & no & yes\\
\bottomrule
\end{tabular}
\caption{Existence of non-equilateral triangles $ABC$ with equilateral $XYZ$, and the side $\ell$ of $XYZ$. Here $\arccos\frac14\approx75.52^\circ$.}
\label{tab:summary}
\end{table}

\subsection{Methods}
The basic idea of the proof of Theorem~A is to factor the differences between the inner products $\ip{X}{Y}$, $\ip{Y}{Z}$ and $\ip{Z}{X}$. For a non-equilateral input, these three quantities agree if and only if a single symmetric expression $J$ vanishes (Lemma~\ref{lem:difference}). The key observation is that $J$ is the imaginary part of the product of two complex numbers whose arguments are transparent. By the solid-angle formula of van Oosterom and Strackee \cite{vOS83}, the first one has argument $-\eps\Omega/2$; the second one is $(1+ix)(1+iy)(1+iz)$, whose argument is $-\pi/2-\Theta$ modulo $2\pi$. The upshot is that $J$ is a positive multiple of $-\cos(\Theta+\eps\Omega/2)$ (Lemma~\ref{lem:J}).

The same complex identity expresses $1+\cos a+\cos b+\cos c$ and the determinant $[A,B,C]$ in terms of $x,y,z$, which yields the quadrics and the sign condition of Theorem~B. The converse direction is achieved by realizing the data through a positive definite Gram matrix, and the threshold $\cos\beta=1/4$ follows from the inequality $e_1e_2\ge9e_3$ between elementary symmetric functions. Finally, in the appendix we give a second proof of the criterion, based on the composition of the rotations about $X,Y,Z$; this is the spherical counterpart of the classical proof of Napoleon's theorem.

The paper is organized as follows. In \S\ref{sec:prelim} we recall the facts we need from spherical trigonometry. In \S\ref{sec:construction} we describe the construction, and in \S\ref{sec:results} we state our results, which are proved in \S\S\ref{sec:proof1}--\ref{sec:proof2}. We give examples with exact data in \S\ref{sec:examples}, and collect some open questions in \S\ref{sec:questions}.

\section{Preliminaries}\label{sec:prelim}

We work on the unit sphere $\Sph\subset\R^3$ and identify each point with its position vector. We write $\ip{P}{Q}$ and $P\times Q$ for the Euclidean inner and cross products, and $[P,Q,R]\defeq\ip{P}{Q\times R}$. Recall that the spherical distance and the chord length are related by $d(P,Q)=\arccos\ip{P}{Q}$ and $|P-Q|=2\sin(d(P,Q)/2)$. For real numbers $x,y,z$ we write $e_1=x+y+z$, $e_2=xy+yz+zx$ and $e_3=xyz$ for their elementary symmetric functions. We refer to \cite{Tod86} for standard spherical trigonometry.

\subsection{Triangles}
Throughout the paper, a spherical triangle $ABC$ has linearly independent vertices, and its sides are the shorter great-circle arcs, of lengths $a=d(B,C)$, $b=d(C,A)$, $c=d(A,B)\in(0,\pi)$. We always take the convex region bounded by these arcs. It lies in an open hemisphere: by linear independence there is a vector $N$ with $\ip{N}{A}=\ip{N}{B}=\ip{N}{C}=1$, and the region lies in $\{U:\ip{N}{U}>0\}$. We order the vertices so that
\begin{equation}\label{eq:delta}
\delta\defeq[A,B,C]>0,
\end{equation}
and we put
\[
p=\cos a,\ q=\cos b,\ r=\cos c,\ \alpha=1+p+q+r,\ g=\cos\tfrac a2\cos\tfrac b2\cos\tfrac c2.
\]
The great circle through $B,C$ separates the sphere into two open hemispheres. We call the hemisphere containing $A$, namely $\{P:\ip{P}{B\times C}>0\}$, the \emph{inner side} of $BC$, and the other one the \emph{outer side}; the same terminology applies cyclically to $CA$ and $AB$. Note that an inward apex lies on the inner side of the relevant great circle, but need not lie inside $ABC$.

The Gram determinant of $A,B,C$ gives
\begin{equation}\label{eq:gram}
\delta^2=1+2pqr-p^2-q^2-r^2 .
\end{equation}
Conversely, let $p,q,r\in(-1,1)$ with $1+2pqr-p^2-q^2-r^2>0$. The matrix
\[
G=\begin{pmatrix}1&r&q\\ r&1&p\\ q&p&1\end{pmatrix}
\]
has leading principal minors $1$, $1-r^2$ and $\det G=1+2pqr-p^2-q^2-r^2$, hence is positive definite. It is therefore the Gram matrix of three unit vectors, and there is a spherical triangle, unique up to isometry, with side cosines $p,q,r$.

By Girard's theorem, the area of $ABC$ is its spherical excess $\Omega=\mathcal A+\mathcal B+\mathcal C-\pi$, where $\mathcal A,\mathcal B,\mathcal C$ are its angles. Since $ABC$ lies in an open hemisphere, we have $0<\Omega<2\pi$.

\subsection{An area formula}

\begin{lemma}\label{lem:area}
For every spherical triangle,
\begin{equation}\label{eq:area}
\alpha+i\delta=4g\,e^{i\Omega/2}.
\end{equation}
\end{lemma}

\begin{proof}
For all real $p,q,r$ we have the polynomial identity
\begin{equation}\label{eq:poly}
2(1+p)(1+q)(1+r)-(1+p+q+r)^2=1+2pqr-p^2-q^2-r^2 .
\end{equation}
Together with \eqref{eq:gram} and $2\cos^2(a/2)=1+p$, it gives $\alpha^2+\delta^2=16g^2$. On the other hand, the solid-angle formula of van Oosterom and Strackee \cite{vOS83} reads
\[
\frac{\Omega}{2}=\operatorname{Arg}(\alpha+i\delta)\in(0,\pi),
\]
the argument lying in $(0,\pi)$ since $\delta>0$. For $\alpha\ne0$ this is the relation $\tan(\Omega/2)=\delta/\alpha$, and for $\alpha=0$ it gives $\Omega=\pi$; the latter case does occur, e.g.~for $p=q=-3/5$ and $r=1/5$. Combining the argument and the modulus, we obtain \eqref{eq:area}. Note that its real and imaginary parts, $\cos(\Omega/2)=\alpha/(4g)$ and $\sin(\Omega/2)=\delta/(4g)$, are the classical half-area formulas of Euler and Cagnoli \cite{Tod86}.
\end{proof}

\subsection{Right triangles and circumcentres}
In a right spherical triangle with hypotenuse $h$, other angles $\mu,\nu$, and leg $m$ opposite to $\mu$, Napier's rules give
\begin{equation}\label{eq:napier}
\cos\mu=\cos m\sin\nu,\ \ \cos h=\cot\mu\cot\nu,\ \ \sin m=\sin h\sin\mu .
\end{equation}

Three points of $\Sph$ not lying on a great circle span a plane $\{U:\ip{N}{U}=t\}$, where $|N|=1$ and $0<t<1$. Their circumcircle is the intersection of this plane with $\Sph$, and its Euclidean centre is $G=tN$. The points $N$ and $-N$ have common distances $\arccos t$ and $\pi-\arccos t$ from the three points, and we call $N$, the one at distance less than $\pi/2$, their \emph{spherical circumcentre}. For a point $P$ of the circle, the tangent vector at $N$ pointing towards $P$ is a positive multiple of $P-\ip{N}{P}N=P-G$. We deduce that the spherical angle at $N$ between the arcs to two points $P,Q$ of the circle equals the Euclidean angle $\angle PGQ$.

\section{The construction}\label{sec:construction}

Recall that $\beta\in(0,\pi/2)$ is fixed.

\begin{definition}
Let $B,C\in\Sph$ be distinct and not antipodal. A point $D\in\Sph$ is a \emph{$\beta$-point} of $BC$ if $|DB|=|DC|$ and the Euclidean angle $\angle BDC$ equals $\beta$.
\end{definition}

Let $M=(B+C)/2$, so that $|M|=\cos(a/2)$ and $|B-M|=\sin(a/2)$. A $\beta$-point $D$ lies on the great circle $\Sigma$ bisecting $BC$ perpendicularly, and the chord triangle $DBC$ has height $|D-M|=\sin(a/2)\cot(\beta/2)$. For $D\in\Sigma$ we have $|D-M|^2=1-2\ip{D}{M}+\cos^2(a/2)$, so that the height condition is equivalent to
\[
\Bigl\langle D,\frac{M}{|M|}\Bigr\rangle=h,\ \text{ where } h=\frac{\cos^2(a/2)-\cos\beta}{(1-\cos\beta)\cos(a/2)} .
\]
Since $0<\cos(a/2)<1$, we have $h<1$, and factoring $h+1$ shows that $h>-1$ if and only if $a<2\beta$. When $-1<h<1$, the two points of $\Sigma$ with this property are exchanged by the reflection in the plane of the great circle $BC$. Thus:
\begin{itemize}[nosep]
\item if $a<2\beta$, there is exactly one $\beta$-point on each side of the great circle $BC$;
\item if $a=2\beta$, the two points coalesce at $-(B+C)/|B+C|$, which lies on that great circle, and the circumcentre at distance less than $\pi/2$ does not exist;
\item if $a>2\beta$, there is no $\beta$-point.
\end{itemize}
From now on we assume that
\begin{equation}\label{eq:domain}
a,b,c<2\beta .
\end{equation}
We let $\eps=1$ for the outward construction and $\eps=-1$ for the inward construction. We let $D,E,F$ be the $\beta$-points of $BC,CA,AB$ on the outer sides if $\eps=1$ (resp.~on the inner sides if $\eps=-1$), and $X,Y,Z$ the spherical circumcentres of $(D,B,C)$, $(E,C,A)$, $(F,A,B)$.

For the side $BC$ we define
\begin{equation}\label{eq:x}
x\defeq\frac{\sqrt{\cos^2(a/2)-\cos^2\beta}}{\cos\beta}, \text{ so that } 1+\cos a=2\cos^2\beta\,(1+x^2),
\end{equation}
and we define $y,z$ from $b,c$ in the same way. The condition \eqref{eq:domain} gives $0<x,y,z<\tan\beta$. Since $x$ is a strictly decreasing function of $a$, two of the coordinates $x,y,z$ are equal if and only if the corresponding sides are.

\begin{lemma}\label{lem:construction}
\begin{enumerate}[label=\textup{(\roman*)}]
\item The circumcentre $X$ is
\begin{equation}\label{eq:X}
X=\frac{x(B+C)-\eps\,B\times C}{\tan\beta\,(1+p)},
\end{equation}
and $\ip{X}{B}=\ip{X}{C}=x\cot\beta$.
\item The spherical triangle $XBC$ is isosceles with apex angle $2\beta$. Its base angle $\theta_a$ is acute and satisfies
\begin{equation}\label{eq:theta}
\cot\theta_a=x,\ \ \sin\theta_a=\frac{\cos\beta}{\cos(a/2)} .
\end{equation}
\item The $\beta$-point is $D=(1+\sec\beta)\,x\cot\beta\,X-\tfrac12\sec\beta\,(B+C)$, and it lies on the same side of $BC$ as $X$.
\end{enumerate}
The analogous statements hold for $Y,E$ and $Z,F$.
\end{lemma}

\begin{proof}
We first define a candidate $X$ by the right-hand side of \eqref{eq:X}, and then show that it is the circumcentre of the chosen $D,B,C$. The vectors $B+C$ and $B\times C$ are orthogonal, with squared lengths $2(1+p)$ and $1-p^2$, so that \eqref{eq:x} gives
\[
|x(B+C)-\eps B\times C|^2=(1+p)(2x^2+1-p)=4\sin^2\beta\cos^2\beta\,(1+x^2)^2=\bigl(\tan\beta\,(1+p)\bigr)^2 .
\]
Thus $|X|=1$, and $\ip{X}{B}=\ip{X}{C}=x\cot\beta\eqdef t\in(0,1)$. Let $\Gamma=\{\ip{X}{U}=t\}\cap\Sph$, with centre $G=tX$ and radius $R=\sqrt{1-t^2}$. Using \eqref{eq:x}, we get
\[
\cos\angle BGC=\frac{\ip{B-G}{C-G}}{R^2}=\frac{p-t^2}{1-t^2}=\cos2\beta,
\]
so that the spherical angle $BXC$ equals $2\beta$.

The component of $X$ in the plane spanned by $B,C$ is a positive multiple of $B+C$. Hence the perpendicular from $X$ to the great circle $BC$ meets the shorter arc at its midpoint $(B+C)/|B+C|$, and divides $XBC$ into two right triangles, each with hypotenuse $d(X,B)$, angle $\beta$ at $X$, and leg $a/2$ opposite to $X$. If $\theta_a$ denotes the angle at $B$, the second rule in \eqref{eq:napier} gives $\cot\beta\cot\theta_a=\cos d(X,B)=x\cot\beta$. Hence $\cot\theta_a=x>0$, and $\theta_a$ is acute. The first rule gives $\cos\beta=\cos(a/2)\sin\theta_a$.

Let $M=(B+C)/2$. Since the central angle is $2\beta<\pi$, we have $|G-M|=R\cos\beta$, so that the point $D'=G+\sec\beta\,(G-M)$ lies on $\Gamma$. It lies on the perpendicular bisector of $BC$, beyond $G$ as seen from $M$, hence on the major arc determined by the chord $BC$. The inscribed angle theorem gives $\angle BD'C=\beta$, and $|D'B|=|D'C|$ by symmetry. Finally, $\ip{X}{B\times C}=-\eps(1-p)/\tan\beta$ and $\ip{D'}{B\times C}=(1+\sec\beta)\,t\,\ip{X}{B\times C}$ both have the sign of $-\eps$. We conclude that $D'$ is the chosen $\beta$-point $D$, and that $X$ is its spherical circumcentre. Expanding the expression for $D'$ gives (iii).
\end{proof}

\begin{figure}[H]
\centering
\begin{tikzpicture}[scale=1.0]
  \coordinate (B) at (0,0);
  \coordinate (C) at (5,0);
  \coordinate (X) at (2.5,1.356);
  \coordinate (D) at (2.5,4.2);
  \coordinate (A) at (1.6,-2.1);
  \draw[densely dotted,gray] (2.5,1.356) circle (2.844);
  \draw[thick] (B) to[bend right=6] (C);
  \draw[thick] (A) to[bend left=6] (B);
  \draw[thick] (A) to[bend right=6] (C);
  \draw (B) -- (X) -- (C);
  \draw[dashed] (B) -- (D) -- (C);
  \draw[dashed,gray] (D) -- (2.5,-0.13);
  \pic[draw,angle radius=5mm,"$2\beta$",angle eccentricity=1.6] {angle=B--X--C};
  \pic[draw,angle radius=9mm,"$\theta_a$",angle eccentricity=1.35] {angle=C--B--X};
  \pic[draw,angle radius=6mm,"$\beta$",angle eccentricity=1.5] {angle=B--D--C};
  \foreach \P/\pos in {B/left,C/right,X/above right,D/above,A/below}
     \fill (\P) circle (1.3pt) node[\pos] {$\P$};
  \node[gray] at (5.4,3.1) {$\Gamma$};
\end{tikzpicture}
\caption{One side of the outward construction, seen from the direction of $X$, so that $X$ and the Euclidean centre $G=tX$ of $\Gamma$ have the same image. The chord triangle $DBC$ (dashed) is isosceles with apex angle $\beta$, and $X$ is the apex of the isosceles spherical triangle on $BC$ with apex angle $2\beta$ and acute base angle $\theta_a$. The picture is schematic: spherical angles are only indicated.}
\label{fig:one-side}
\end{figure}

\begin{remark}\label{rem:equivalent}
Lemma~\ref{lem:construction} provides an intrinsic description of the construction: on each side of $ABC$, erect the isosceles spherical triangle with apex angle $2\beta$ and acute base angles, the three apices being taken all outwards or all inwards. We stress that the condition on the base angles cannot be dropped: on each side of $BC$ there is also an isosceles triangle with the same apex angle and obtuse base angles, whose apex is $(-x(B+C)-\eps B\times C)/(\tan\beta\,(1+p))$. In the plane the base angle equals $\pi/2-\beta$. On the sphere, \eqref{eq:theta} instead gives $\sin\theta_a>\cos\beta$, and therefore
\begin{equation}\label{eq:theta-bound}
\theta_a>\frac\pi2-\beta \text{ and } \Theta=\theta_a+\theta_b+\theta_c>\frac{3\pi}2-3\beta .
\end{equation}
For $\beta=\pi/3$ we have $\sec\beta=2$, and Lemma~\ref{lem:construction}(iii) gives $D+B+C=3x\cot\beta\,X$. Hence $X=(D+B+C)/|D+B+C|$ is the centroid of the equilateral spherical triangle $DBC$, which is the centre used in \cite{DNV24}.
\end{remark}

\section{Main results}\label{sec:results}

By \eqref{eq:theta}, we have $1+ix=i\,e^{-i\theta_a}/\sin\theta_a$, and similarly for $y,z$. Multiplying these three identities, we obtain
\begin{equation}\label{eq:product}
(1-e_2)+i(e_1-e_3)=(1+ix)(1+iy)(1+iz)=-\frac{\sin\Theta+i\cos\Theta}{\sin\theta_a\sin\theta_b\sin\theta_c}.
\end{equation}
Since $0<\Theta<3\pi/2$, this shows that $e_1-e_3>0$ if and only if $\Theta>\pi/2$, and that $e_2=1$ if and only if $\Theta=\pi$.

\begin{theorem}\label{thm:angle}
Let $ABC$ be a non-equilateral spherical triangle satisfying \eqref{eq:domain}, let $X,Y,Z$ be obtained by the construction with sign $\eps$, and put $\Phi=\Theta+\eps\Omega/2$. Then $XYZ$ is equilateral if and only if $\cos\Phi=0$. In this case, let $\sigma=\sin\Phi\in\{1,-1\}$. If $\sigma=1$, then $\beta>\pi/3$. The triangle $XYZ$ is nondegenerate, and its side $\ell$ satisfies
\begin{equation}\label{eq:side}
\cos\ell=\frac{\sigma\cos\beta}{1-\sigma\cos\beta}.
\end{equation}
\end{theorem}

\begin{remark}\label{rem:branches}
Since $0<\theta_a,\theta_b,\theta_c<\pi/2$ and $0<\Omega<2\pi$, the condition $\cos\Phi=0$ means that
\begin{equation}\label{eq:branches}
\Theta+\tfrac12\Omega\in\{\tfrac12\pi,\tfrac32\pi\} \text{ (outward)}, \ \ \Theta-\tfrac12\Omega\in\{-\tfrac12\pi,\tfrac12\pi\} \text{ (inward)}.
\end{equation}
The value $\pi/2$ corresponds to $\sigma=1$, and the other value in each case to $\sigma=-1$. Under this condition, one reads off from \eqref{eq:branches} that $\Theta>\pi/2$ if and only if $\eps\sigma=-1$.
\end{remark}

For $\kappa\in\{\cos\beta,-\cos\beta\}$, we let $\cQ_\kappa\subset\R^3$ be the quadric
\begin{equation}\label{eq:quadric}
\kappa^2(x^2+y^2+z^2)+2\kappa^3(xy+yz+zx)=(1-\kappa)^2(1+2\kappa).
\end{equation}

\begin{theorem}\label{thm:quadric}
\begin{enumerate}[label=\textup{(\alph*)}]
\item Let $ABC$ be a non-equilateral spherical triangle satisfying \eqref{eq:domain}. The construction with sign $\eps$ gives an equilateral triangle $XYZ$ if and only if, for some $\kappa\in\{\cos\beta,-\cos\beta\}$,
\[
(x,y,z)\in\cQ_\kappa \text{ and } \eps\kappa(x+y+z-xyz)>0 .
\]
In this case $\kappa=-\sigma\cos\beta$ and $\cos\ell=-\kappa/(1+\kappa)$.
\item Conversely, let $(x,y,z)\in\cQ_\kappa$ with $0<x,y,z<\tan\beta$, not all equal, and $x+y+z\ne xyz$. Then there is a spherical triangle, unique up to isometry, with side cosines $2\cos^2\beta\,(1+x^2)-1$, $2\cos^2\beta\,(1+y^2)-1$, $2\cos^2\beta\,(1+z^2)-1$. It satisfies \eqref{eq:domain}, and its construction with $\eps=\sgn(\kappa(x+y+z-xyz))$ gives an equilateral $XYZ$.
\end{enumerate}
\end{theorem}

Note that, among the triples coming from triangles, the diagonal $x=y=z$ corresponds exactly to the equilateral ones, and that on $\cQ_\kappa$ the surface $x+y+z=xyz$ is where the Gram determinant of the side cosines in (b) vanishes (see the proof). For each quadric, the sign of $x+y+z-xyz$ thus determines which construction gives an equilateral triangle; some of the resulting parts are empty in the admissible cube. The side of $XYZ$ is $\arccos(-\cos\beta/(1+\cos\beta))$ when the construction selects $\kappa=\cos\beta$, and $\arccos(\cos\beta/(1-\cos\beta))$ when it selects $\kappa=-\cos\beta$; the latter case requires $\beta>\pi/3$.

\begin{corollary}\label{cor:small}
Let $0<\beta\le\pi/3$, and let $ABC$ be non-equilateral and satisfy \eqref{eq:domain}.
\begin{enumerate}[label=\textup{(\roman*)}]
\item The inward construction never gives an equilateral $XYZ$.
\item The outward construction gives one if and only if $\Theta+\Omega/2=3\pi/2$, i.e.~if and only if $(x,y,z)\in\cQ_{\cos\beta}$. In this case $\cos\ell=-\cos\beta/(1+\cos\beta)$.
\item The quadric $\cQ_{\cos\beta}$ is an ellipsoid, and each of its points in the open positive octant, except the diagonal one, comes from a triangle as in (ii). In particular, these triangles form a two-parameter family.
\end{enumerate}
\end{corollary}

For $\beta=\pi/3$ we have $x=\sqrt{1+2\cos a}$, the ellipsoid becomes $x^2+y^2+z^2+xy+yz+zx=2$, and $\cos\ell=-1/3$. This is the classification of \cite{DNV24}.

\begin{corollary}\label{cor:large}
Let $\pi/3<\beta<\pi/2$. Within the cube $(0,\tan\beta)^3$, consider the four sets
\[
\cQ_{\cos\beta}\cap\{e_1>e_3\},\ \cQ_{-\cos\beta}\cap\{e_1<e_3\},\ \cQ_{-\cos\beta}\cap\{e_1>e_3\},\ \cQ_{\cos\beta}\cap\{e_1<e_3\},
\]
which correspond to the outward, outward, inward and inward construction respectively. The first three contain a two-dimensional set of non-diagonal points, and the last one contains non-diagonal points if and only if $\cos\beta<1/4$. In particular, the outward construction gives equilateral triangles of two different sizes, and so does the inward one when $\cos\beta<1/4$.
\end{corollary}

\begin{remark}\label{rem:both}
A non-equilateral $ABC$ gives equilateral triangles of centres in both constructions if and only if $\Theta=\Omega=\pi$. Indeed, by \eqref{eq:branches}, the outward and inward values of $\Phi$ lie in $\{\pi/2,3\pi/2\}$ and $\{-\pi/2,\pi/2\}$ respectively. Their difference is $\Omega\in(0,2\pi)$ and their sum is $2\Theta>0$, so that they must be $3\pi/2$ and $\pi/2$; the converse follows from Theorem~\ref{thm:angle}. By Lemma~\ref{lem:area} and \eqref{eq:product}, this condition is equivalent to $\alpha=0$ and $e_2=1$. To express it on the quadrics, put $c_\beta=\cos\beta$. The algebraic identity \eqref{eq:alphaQ}, proved independently below, gives
\[
2Q_{c_\beta}=\alpha-4c_\beta^3(1-e_2),\ \ 2Q_{-c_\beta}=\alpha+4c_\beta^3(1-e_2),
\]
so that the condition is also equivalent to $(x,y,z)\in\cQ_{\cos\beta}\cap\cQ_{-\cos\beta}$. Since the inward value of $\Phi$ gives $\sigma=1$, this forces $\beta>\pi/3$. Example~\ref{ex:both} provides an explicit instance.
\end{remark}

\begin{proposition}\label{prop:DEF}
Let $ABC$ be non-equilateral and satisfy \eqref{eq:domain}. Assume that the chosen construction gives an equilateral triangle $XYZ$ whose side $\ell$ is greater than $\pi/2$. Then $DEF$ is a nondegenerate equilateral triangle, whose side $m$ satisfies
\begin{equation}\label{eq:DEF}
\cos m=\frac{\cos\beta\,(1-\cos\beta+e_2\cos\beta)}{1-\cos\beta}
\end{equation}
and $0<m<\beta$.
\end{proposition}

\begin{remark}\label{rem:DEF}
\begin{enumerate}[label=(\arabic*)]
\item The hypothesis that $ABC$ is not equilateral cannot be omitted. Take $\beta=\pi/3$ and $\cos a=\cos b=\cos c=-1/3$, so that $ABC$ is a face of a regular tetrahedron. Then $x=y=z=1/\sqrt3$ lies on $\cQ_{1/2}$, and the outward construction gives $X=-A$, $Y=-B$, $Z=-C$. The triangle $XYZ$ is equilateral with side greater than $\pi/2$, but $D=E=F=-(A+B+C)$ and $m=0$.
\item The converse fails, even for a non-equilateral input. For $\beta=\pi/3$, take $\cos a=\cos b=-1/5$, $\cos c=7/10$, and construct outwards. This is a valid triangle with $(x,y,z)=\sqrt{3/5}\,(1,1,2)$. We find $\ip{X}{Y}=-1/2$ and $\ip{Y}{Z}=\ip{Z}{X}=0$, so that the sides of $XYZ$ are $120^\circ,90^\circ,90^\circ$. Nevertheless, $DEF$ is equilateral, with $\ip{D}{E}=\ip{E}{F}=\ip{F}{D}=-1/5$.
\item By Theorem~\ref{thm:angle}, the condition $\ell>\pi/2$ is equivalent to $\sigma=-1$, i.e.~the chosen construction selects $\kappa=\cos\beta$ in Theorem~\ref{thm:quadric}. We stress that mere membership in $\cQ_{\cos\beta}$ would not suffice. In Example~\ref{ex:both} the parameter point lies on both quadrics, but the inward construction has $\ell<\pi/2$ and a non-equilateral $DEF$; Example~\ref{ex:second} gives another equilateral $XYZ$ with $\ell<\pi/2$ and non-equilateral $DEF$. Note also that, contrary to $\ell$, the side $m$ varies within each family covered by the proposition.
\end{enumerate}
\end{remark}

\section{Proof of Theorem~\ref{thm:angle}}\label{sec:proof1}

\begin{lemma}\label{lem:inner}
Let $N_{xy}=\alpha xy-\eps\delta(x+y)+pq-r$, and define $N_{yz}$, $N_{zx}$ cyclically; for instance $N_{zx}=\alpha zx-\eps\delta(z+x)+rp-q$. Then
\begin{equation}\label{eq:XY}
\ip{X}{Y}=\frac{N_{xy}}{\tan^2\beta\,(1+p)(1+q)}=\frac{N_{xy}}{4\sin^2\beta\cos^2\beta\,(1+x^2)(1+y^2)},
\end{equation}
and similarly for the other pairs.
\end{lemma}

\begin{proof}
We take the inner product of the numerator in \eqref{eq:X} with its cyclic counterpart. The four terms are given by $\ip{B+C}{C+A}=\alpha$, $\ip{B+C}{C\times A}=\ip{B\times C}{C+A}=\delta$ and $\ip{B\times C}{C\times A}=pq-r$, which gives the first expression in \eqref{eq:XY}; the second one follows from \eqref{eq:x}.
\end{proof}

\begin{lemma}\label{lem:difference}
Let $J=\alpha(e_1-e_3)-\eps\delta(1-e_2)$. Then
\[
(1+z^2)N_{xy}-(1+y^2)N_{zx}=(y-z)J,
\]
and the same holds after cyclically permuting $x,y,z$ together with $a,b,c$.
\end{lemma}

\begin{proof}
We compare the terms of $N_{xy}$ and $N_{zx}$. The terms involving $\alpha$ give $xy(1+z^2)-xz(1+y^2)=x(y-z)(1-yz)$, and those involving $\delta$ give $(x+y)(1+z^2)-(x+z)(1+y^2)=(y-z)(1-e_2)$. For the remaining terms, recall from \eqref{eq:x} that $1+z^2=(1+r)/(2\cos^2\beta)$ and $1+y^2=(1+q)/(2\cos^2\beta)$, and note the identity
\[
(1+r)(pq-r)-(1+q)(rp-q)=(q-r)\alpha .
\]
Since $q-r=2\cos^2\beta\,(y^2-z^2)$, these terms contribute $\alpha(y-z)(y+z)$. The total is $(y-z)[\alpha(x(1-yz)+y+z)-\eps\delta(1-e_2)]=(y-z)J$.
\end{proof}

\begin{lemma}\label{lem:J}
Let $S=\sin\theta_a\sin\theta_b\sin\theta_c$ and $\Phi=\Theta+\eps\Omega/2$. Then
\[
(\alpha-i\eps\delta)\bigl((1-e_2)+i(e_1-e_3)\bigr)=-\frac{4g}{S}(\sin\Phi+i\cos\Phi).
\]
In particular, $J=-(4g/S)\cos\Phi$.
\end{lemma}

\begin{proof}
By Lemma~\ref{lem:area} we have $\alpha-i\eps\delta=4g\,e^{-i\eps\Omega/2}$, and by \eqref{eq:product} we have $(1-e_2)+i(e_1-e_3)=-i\,e^{-i\Theta}/S$. Their product is $-i(4g/S)e^{-i\Phi}$, whose imaginary part is $J$.
\end{proof}

\begin{lemma}\label{lem:identities}
For $\kappa\in\{\cos\beta,-\cos\beta\}$, let $Q_\kappa=\kappa^2(x^2+y^2+z^2)+2\kappa^3e_2-(1-\kappa)^2(1+2\kappa)$. Then
\begin{equation}\label{eq:alphaQ}
\alpha-4\kappa^3(1-e_2)=2Q_\kappa .
\end{equation}
If moreover $\eps\delta=4\kappa^3(e_1-e_3)$, then
\begin{equation}\label{eq:NQ}
N_{xy}+4\kappa^3(1-\kappa)(1+x^2)(1+y^2)=2(xy-1)Q_\kappa .
\end{equation}
\end{lemma}

\begin{proof}
By \eqref{eq:x}, we have $\alpha=2\kappa^2(3+x^2+y^2+z^2)-2$, and \eqref{eq:alphaQ} reduces to $-2+6\kappa^2-4\kappa^3=-2(1-\kappa)^2(1+2\kappa)$.

For \eqref{eq:NQ}, we substitute $\alpha=4\kappa^3(1-e_2)+2Q_\kappa$ and $\eps\delta=4\kappa^3(e_1-e_3)$. Since $(1-e_2)xy-(e_1-e_3)(x+y)=xy-x^2y^2-(x+y)^2-z(x+y)$, we get
\[
N_{xy}=2xyQ_\kappa+4\kappa^3\bigl(xy-x^2y^2-(x+y)^2\bigr)-4\kappa^3z(x+y)+pq-r .
\]
The variable $z$ also occurs in $-r=1-2\kappa^2(1+z^2)$, and
\[
-4\kappa^3z(x+y)-2\kappa^2z^2=-2Q_\kappa-2(1-\kappa)^2(1+2\kappa)+2\kappa^2(x^2+y^2)+4\kappa^3xy .
\]
After this substitution, $N_{xy}-2(xy-1)Q_\kappa$ no longer involves $z$, and equals
\[
4\kappa^3\bigl(2xy-x^2y^2-(x+y)^2\bigr)+pq+1-2\kappa^2-2(1-\kappa)^2(1+2\kappa)+2\kappa^2(x^2+y^2).
\]
Substituting $p=2\kappa^2(1+x^2)-1$ and $q=2\kappa^2(1+y^2)-1$, this expression reduces to $-4\kappa^3(1-\kappa)(1+x^2)(1+y^2)$, which proves \eqref{eq:NQ}.
\end{proof}

\begin{proof}[Proof of Theorem~\ref{thm:angle}]
By Lemmas~\ref{lem:inner} and~\ref{lem:difference}, the equality $\ip{X}{Y}=\ip{X}{Z}$ is equivalent to $(y-z)J=0$, and the other two equalities are equivalent to $(z-x)J=0$ and $(x-y)J=0$. Since $ABC$ is not equilateral, the coordinates $x,y,z$ are not all equal. We deduce that the three pairwise distances are equal if and only if $J=0$, i.e.~by Lemma~\ref{lem:J}, if and only if $\cos\Phi=0$.

We now assume that $\cos\Phi=0$, and let $\sigma=\sin\Phi$. By Lemma~\ref{lem:J}, the product $(\alpha-i\eps\delta)((1-e_2)+i(e_1-e_3))=-4\sigma g/S$ is real. Since the second factor has squared modulus $(1+x^2)(1+y^2)(1+z^2)=S^{-2}$, we get
\[
\alpha-i\eps\delta=-4\sigma gS\bigl((1-e_2)-i(e_1-e_3)\bigr).
\]
By \eqref{eq:theta}, we have $gS=\cos^3\beta$. Letting $\kappa=-\sigma\cos\beta$, this reads
\begin{equation}\label{eq:alpha-delta}
\alpha=4\kappa^3(1-e_2), \ \ \eps\delta=4\kappa^3(e_1-e_3).
\end{equation}
By \eqref{eq:alphaQ}, we have $Q_\kappa=0$, and \eqref{eq:NQ} and \eqref{eq:XY} then give
\[
\ip{X}{Y}=\frac{-4\kappa^3(1-\kappa)}{4\kappa^2(1-\kappa^2)}=-\frac{\kappa}{1+\kappa}\eqdef\rho ,
\]
and the same value for the other two pairs. This is \eqref{eq:side}.

It remains to show that $XYZ$ is nondegenerate. If $\sigma=1$, then $\rho=\cos\beta/(1-\cos\beta)\le1$ forces $\cos\beta\le1/2$. Equality would give $X=Y=Z$; since $\ip{X}{C}=x\cot\beta$ and $\ip{Y}{C}=y\cot\beta$, this would force $x=y$, and similarly $y=z$, which is impossible. Hence $\beta>\pi/3$ and $0<\rho<1$. If $\sigma=-1$, then $-1/2<\rho<0$. In both cases the Gram matrix of $X,Y,Z$, whose eigenvalues are $1-\rho,1-\rho,1+2\rho$, is positive definite, and we conclude that $XYZ$ is nondegenerate.
\end{proof}

\section{Proofs of Theorem~\ref{thm:quadric} and its consequences}\label{sec:proof2}

\begin{proof}[Proof of Theorem~\ref{thm:quadric}]
(a) If $XYZ$ is equilateral, the proof of Theorem~\ref{thm:angle} gives $Q_\kappa=0$ and $\eps\delta=4\kappa^3(e_1-e_3)$ with $\kappa=-\sigma\cos\beta$. Since $\delta>0$, the numbers $\eps$ and $\kappa(e_1-e_3)$ have the same sign.

Conversely, assume that $Q_\kappa=0$ and $\eps\kappa(e_1-e_3)>0$. By \eqref{eq:alphaQ}, we have $\alpha=4\kappa^3(1-e_2)$. On the other hand, Lemma~\ref{lem:area}, \eqref{eq:x} and \eqref{eq:product} give
\[
\alpha^2+\delta^2=2(1+p)(1+q)(1+r)=16\kappa^6(1+x^2)(1+y^2)(1+z^2)=16\kappa^6\bigl((1-e_2)^2+(e_1-e_3)^2\bigr).
\]
Hence $\delta=4|\kappa^3(e_1-e_3)|$, and the sign condition gives $\eps\delta=4\kappa^3(e_1-e_3)$. It follows that $J=0$, and $XYZ$ is equilateral by the proof of Theorem~\ref{thm:angle}.

(b) We define $p,q,r$ by the stated formulas. The bounds $1<1+x^2<\sec^2\beta$ give $-1<p<1$, and similarly for $q,r$. To establish the existence of a triangle, we combine \eqref{eq:poly} and \eqref{eq:alphaQ} with the algebraic identity $(1+x^2)(1+y^2)(1+z^2)=(1-e_2)^2+(e_1-e_3)^2$ from \eqref{eq:product}. This gives
\[
1+2pqr-p^2-q^2-r^2=2(1+p)(1+q)(1+r)-\alpha^2=16\kappa^6(e_1-e_3)^2>0 .
\]
The Gram criterion of \S\ref{sec:prelim} now provides a triangle with these side cosines, unique up to isometry. The positivity of $x,y,z$ gives \eqref{eq:domain}, and the coordinates computed from its sides are the original $x,y,z$. Since they are not all equal, the triangle is not equilateral, and we conclude by applying (a) with $\eps=\sgn(\kappa(e_1-e_3))$.
\end{proof}

We need three elementary facts about the quadrics. Their quadratic form has the eigenvalue $\kappa^2(1+2\kappa)$ in the direction $(1,1,1)$, and the eigenvalue $\kappa^2(1-\kappa)$, with multiplicity two, on its orthogonal complement. We put $K_\kappa=(1-\kappa)^2(1+2\kappa)/\kappa^2$. The inequality $x^2+y^2+z^2\ge e_2$ shows that, for $\kappa>0$, every point of $\cQ_\kappa$ satisfies
\begin{equation}\label{eq:e2-bound}
e_2\le\frac{K_\kappa}{1+2\kappa}=\frac{(1-\kappa)^2}{\kappa^2},
\end{equation}
with equality only on the diagonal. Finally, for $\kappa=\cos\beta$ we have
\begin{equation}\label{eq:K-bound}
K_\kappa<\tan^2\beta=\frac{1-\kappa^2}{\kappa^2},
\end{equation}
since this is equivalent to $(1-\kappa)(1+2\kappa)<1+\kappa$, i.e.~to $2\kappa^2>0$. As a point of $\cQ_{\cos\beta}$ in the open positive octant satisfies $x^2<K_\kappa$, the positive part of $\cQ_{\cos\beta}$ lies in the cube $(0,\tan\beta)^3$.

\begin{proof}[Proof of Corollary~\ref{cor:small}]
By \eqref{eq:theta-bound} we have $\Theta>\pi/2$, and $\Omega/2\in(0,\pi)$. For the outward construction we get $\Phi>\pi/2$, so that by \eqref{eq:branches} only $\Phi=3\pi/2$ is possible, with $\sigma=-1$, i.e.~$\kappa=\cos\beta$; this gives (ii). For the inward construction we get $\Phi>-\pi/2$, so that only $\Phi=\pi/2$ is possible, with $\sigma=1$, which is excluded by Theorem~\ref{thm:angle} since $\beta\le\pi/3$; this gives (i).

For (iii), let $\kappa=\cos\beta\in[1/2,1)$. Both eigenvalues and $K_\kappa$ are positive, so that $\cQ_\kappa$ is an ellipsoid. A point of it in the open positive octant lies in the cube by \eqref{eq:K-bound}. As $\tan\beta\le\sqrt3$, each base angle exceeds $\pi/6$, so that $\Theta>\pi/2$ and $e_1>e_3$. If the point is not diagonal, Theorem~\ref{thm:quadric}(b) provides a triangle, and the construction is the outward one since $\kappa(e_1-e_3)>0$. Finally, near $(0,0,\sqrt{K_\kappa})$ the quadric $\cQ_\kappa$ is the graph of
\begin{equation}\label{eq:zxy}
z(x,y)=-\kappa(x+y)+\sqrt{K_\kappa-x^2-y^2-2\kappa xy+\kappa^2(x+y)^2},
\end{equation}
and $z(x,y)>0$ for small $x,y\ge0$. Taking small $x,y>0$ with $x\ne y$ gives a two-parameter set of non-diagonal points of $\cQ_\kappa$ in the open octant.
\end{proof}

\begin{proof}[Proof of Corollary~\ref{cor:large}]
In this proof we write $c=\cos\beta\in(0,1/2)$. For $\kappa=\pm c$, the two eigenvalues and $K_\kappa$ are positive, so that both quadrics are ellipsoids. The strict inequalities defining the four parts give relatively open subsets of these surfaces. We show that the first three parts are nonempty in the admissible cube, and determine when the fourth one is.

$\cQ_c\cap\{e_1>e_3\}$: by \eqref{eq:K-bound} we have $0<\sqrt{K_c}<\tan\beta$. For small $x,y>0$ with $x\ne y$, the point $(x,y,z(x,y))$ given by \eqref{eq:zxy} with $\kappa=c$ lies on $\cQ_c$ and in the open cube, and $e_1-e_3$ is close to $\sqrt{K_c}>0$.

$\cQ_{-c}\cap\{e_1>e_3\}$: here $K_{-c}=(1+c)^2(1-2c)/c^2>0$, and $K_{-c}<\tan^2\beta$ is equivalent to $(1+c)(1-2c)<1-c$, i.e.~to $2c^2>0$. We then use \eqref{eq:zxy} with $\kappa=-c$ near $(0,0,\sqrt{K_{-c}})$ as before.

$\cQ_{-c}\cap\{e_1<e_3\}$: the positive diagonal point of $\cQ_{-c}$ is $(t,t,t)$, where $t=(1+c)/(\sqrt3\,c)$. Since $c<1/2$, we have $t>\sqrt3$ and $e_1-e_3=t(3-t^2)<0$. Moreover $t<\tan\beta$, because $(1+c)^2<3(1-c^2)$. A sufficiently small neighbourhood of this point on the ellipsoid remains in the open cube and satisfies $e_1<e_3$, and removing the diagonal point leaves a two-dimensional set of admissible points.

$\cQ_c\cap\{e_1<e_3\}$: if $c\ge1/4$, then \eqref{eq:e2-bound} gives $e_2\le(1-c)^2/c^2\le9$, and the arithmetic--geometric mean inequality $e_1e_2\ge9e_3$, with equality only for $x=y=z$, gives $e_3\le e_1e_2/9\le e_1$; hence there are no non-diagonal points. If $c<1/4$, the diagonal point $t=(1-c)/(\sqrt3\,c)>\sqrt3$ of $\cQ_c$ satisfies $e_1<e_3$ and lies in the open cube by \eqref{eq:K-bound}, and the argument of the previous case applies.

The statement about the two sizes follows from Theorem~\ref{thm:quadric}(a).
\end{proof}

\begin{proof}[Proof of Proposition~\ref{prop:DEF}]
The inequality $\cos\ell<0$ forces $\sigma=-1$ in \eqref{eq:side}. We may therefore take $\kappa=\cos\beta$ in \eqref{eq:alpha-delta}, and $\ip{X}{Y}=-\kappa/(1+\kappa)$. Formula~\eqref{eq:X} gives $\ip{X}{C+A}=(x\alpha-\eps\delta)/(\tan\beta\,(1+p))$. Using \eqref{eq:alpha-delta} and the identity $x(1-e_2)-(e_1-e_3)=-(1+x^2)(y+z)$, we obtain $x\alpha-\eps\delta=-4\kappa^3(1+x^2)(y+z)$, and hence
\[
\ip{X}{C+A}=-2\kappa(y+z)\cot\beta, \text{ and similarly } \ip{Y}{B+C}=-2\kappa(x+z)\cot\beta .
\]
Using Lemma~\ref{lem:construction}(iii), $\sec\beta=1/\kappa$ and $\cot^2\beta=\kappa^2/(1-\kappa^2)$, we get
\begin{align*}
\ip{D}{E}&=\frac{(1+\kappa)^2}{\kappa^2}xy\cot^2\beta\,\ip{X}{Y}-\frac{1+\kappa}{2\kappa^2}\cot\beta\bigl(x\ip{X}{C+A}+y\ip{Y}{B+C}\bigr)+\frac{\alpha}{4\kappa^2}\\
&=\frac{\kappa}{1-\kappa}\bigl(-xy+x(y+z)+y(x+z)\bigr)+\kappa(1-e_2)=\kappa+\frac{\kappa^2e_2}{1-\kappa}.
\end{align*}
This expression is symmetric in $x,y,z$, so that the three pairwise inner products of $D,E,F$ are equal, with the value given in \eqref{eq:DEF}. Since $e_2>0$, their common value $\rho_D$ is greater than $\cos\beta$. Since $ABC$ is not equilateral, the inequality \eqref{eq:e2-bound} is strict, and it gives $\rho_D<1$. Thus $\cos\beta<\rho_D<1$, and the Gram eigenvalues $1-\rho_D,1-\rho_D,1+2\rho_D$ are positive. We conclude that $DEF$ is nondegenerate, and that its side satisfies $0<m<\beta$.
\end{proof}

We point out that the choice $\kappa=\cos\beta$ is used in the equality $\sec\beta=1/\kappa$. For the other choice of $\kappa$, the same computation does not yield a symmetric expression for $\ip{D}{E}$.

\section{Examples}\label{sec:examples}

The following examples illustrate, in turn, the classical spherical case, an input common to both constructions, and the two additional families of Table~\ref{tab:summary}. The numerical values were obtained from \eqref{eq:X} and Lemma~\ref{lem:construction}(iii). All angles, including the area $\Omega$ expressed as a spherical excess, are given in degrees.

\begin{example}[$\beta=\pi/3$]\label{ex:DNV}
Take $(x,y,z)=\frac{\sqrt2}{5}(1,2,3)$ on the ellipsoid $x^2+y^2+z^2+xy+yz+zx=2$. Then $e_2=22/25$, and the side cosines are $-23/50$, $-17/50$, $-7/50$, so that the sides are approximately $117.39^\circ$, $109.88^\circ$, $98.05^\circ$. This is the triangle of \cite[Example~1.6]{DNV24}. Its base angles are approximately $74.21^\circ$, $60.50^\circ$, $49.68^\circ$, so that $\Theta\approx184.40^\circ$, and $\Omega\approx171.21^\circ$; the exact relation is $\Theta+\Omega/2=270^\circ$. The outward centres have common inner product $-1/3$, and $\cos m=(1+e_2)/2=47/50$. We refer to Figure~\ref{fig:dnv} for a picture.
\end{example}

\begin{figure}[H]
\centering
\begin{tikzpicture}[scale=1.3]
\draw[thick] (-0.965,-0.427) -- (-0.920,-0.470) -- (-0.874,-0.511) -- (-0.827,-0.549) -- (-0.779,-0.585) -- (-0.730,-0.617) -- (-0.680,-0.648) -- (-0.630,-0.675) -- (-0.579,-0.701) -- (-0.528,-0.723) -- (-0.477,-0.744) -- (-0.425,-0.762) -- (-0.372,-0.778) -- (-0.320,-0.792) -- (-0.267,-0.804) -- (-0.214,-0.813) -- (-0.161,-0.821) -- (-0.108,-0.826) -- (-0.055,-0.829) -- (-0.002,-0.831) -- (0.051,-0.830) -- (0.104,-0.826) -- (0.157,-0.821) -- (0.210,-0.814) -- (0.263,-0.805) -- (0.316,-0.793) -- (0.368,-0.780) -- (0.421,-0.764) -- (0.473,-0.746) -- (0.524,-0.725);
\draw[thick] (0.524,-0.725) -- (0.528,-0.660) -- (0.531,-0.597) -- (0.531,-0.535) -- (0.529,-0.476) -- (0.526,-0.418) -- (0.521,-0.362) -- (0.515,-0.308) -- (0.508,-0.255) -- (0.499,-0.202) -- (0.489,-0.151) -- (0.478,-0.101) -- (0.466,-0.052) -- (0.453,-0.004) -- (0.439,0.044) -- (0.424,0.091) -- (0.407,0.138) -- (0.390,0.184) -- (0.371,0.230) -- (0.351,0.275) -- (0.330,0.320) -- (0.308,0.365) -- (0.284,0.410) -- (0.259,0.455) -- (0.232,0.500) -- (0.204,0.545) -- (0.174,0.591) -- (0.142,0.636) -- (0.107,0.681) -- (0.071,0.727);
\draw[thick] (0.071,0.727) -- (0.019,0.709) -- (-0.031,0.689) -- (-0.080,0.668) -- (-0.128,0.645) -- (-0.175,0.621) -- (-0.220,0.595) -- (-0.265,0.568) -- (-0.308,0.540) -- (-0.350,0.511) -- (-0.392,0.480) -- (-0.432,0.448) -- (-0.471,0.414) -- (-0.509,0.379) -- (-0.547,0.343) -- (-0.583,0.305) -- (-0.618,0.265) -- (-0.653,0.224) -- (-0.686,0.181) -- (-0.718,0.137) -- (-0.750,0.090) -- (-0.780,0.042) -- (-0.808,-0.008) -- (-0.836,-0.061) -- (-0.862,-0.116) -- (-0.886,-0.173) -- (-0.909,-0.232) -- (-0.930,-0.294) -- (-0.949,-0.359) -- (-0.965,-0.427);
\draw[thin,gray] (1.689,0.551) -- (1.677,0.462) -- (1.660,0.376) -- (1.638,0.293) -- (1.613,0.214) -- (1.584,0.139) -- (1.552,0.068) -- (1.517,0.000) -- (1.479,-0.063) -- (1.440,-0.124) -- (1.399,-0.180) -- (1.357,-0.233) -- (1.313,-0.283) -- (1.268,-0.329) -- (1.223,-0.373) -- (1.177,-0.413) -- (1.130,-0.451) -- (1.083,-0.485) -- (1.036,-0.517) -- (0.989,-0.547) -- (0.941,-0.574) -- (0.894,-0.599) -- (0.847,-0.621) -- (0.800,-0.642) -- (0.753,-0.660) -- (0.707,-0.677) -- (0.661,-0.691) -- (0.615,-0.704) -- (0.569,-0.715) -- (0.524,-0.725);
\draw[thin,gray] (1.689,0.551) -- (1.615,0.602) -- (1.541,0.647) -- (1.468,0.686) -- (1.396,0.720) -- (1.326,0.750) -- (1.256,0.775) -- (1.188,0.797) -- (1.122,0.815) -- (1.057,0.830) -- (0.994,0.841) -- (0.932,0.850) -- (0.872,0.857) -- (0.814,0.861) -- (0.757,0.863) -- (0.702,0.863) -- (0.649,0.862) -- (0.597,0.858) -- (0.546,0.853) -- (0.497,0.847) -- (0.449,0.840) -- (0.402,0.831) -- (0.357,0.821) -- (0.313,0.811) -- (0.270,0.799) -- (0.228,0.786) -- (0.187,0.773) -- (0.148,0.758) -- (0.109,0.743) -- (0.071,0.727);
\draw[thin,gray] (-1.308,1.175) -- (-1.230,1.182) -- (-1.155,1.185) -- (-1.083,1.185) -- (-1.014,1.182) -- (-0.947,1.176) -- (-0.883,1.169) -- (-0.822,1.160) -- (-0.763,1.148) -- (-0.706,1.136) -- (-0.651,1.122) -- (-0.599,1.107) -- (-0.549,1.090) -- (-0.500,1.073) -- (-0.454,1.055) -- (-0.409,1.036) -- (-0.366,1.017) -- (-0.325,0.997) -- (-0.285,0.976) -- (-0.246,0.955) -- (-0.209,0.934) -- (-0.173,0.912) -- (-0.139,0.890) -- (-0.106,0.867) -- (-0.074,0.844) -- (-0.043,0.821) -- (-0.013,0.798) -- (0.016,0.775) -- (0.044,0.751) -- (0.071,0.727);
\draw[thin,gray] (-1.308,1.175) -- (-1.338,1.102) -- (-1.363,1.029) -- (-1.384,0.957) -- (-1.401,0.884) -- (-1.413,0.813) -- (-1.422,0.742) -- (-1.427,0.672) -- (-1.428,0.604) -- (-1.426,0.537) -- (-1.421,0.471) -- (-1.413,0.407) -- (-1.403,0.345) -- (-1.390,0.285) -- (-1.375,0.226) -- (-1.357,0.170) -- (-1.338,0.115) -- (-1.316,0.062) -- (-1.293,0.012) -- (-1.269,-0.037) -- (-1.243,-0.084) -- (-1.216,-0.130) -- (-1.188,-0.173) -- (-1.158,-0.214) -- (-1.128,-0.254) -- (-1.097,-0.292) -- (-1.065,-0.328) -- (-1.032,-0.363) -- (-0.999,-0.396) -- (-0.965,-0.427);
\draw[thin,gray] (-0.371,-1.729) -- (-0.425,-1.688) -- (-0.476,-1.645) -- (-0.524,-1.602) -- (-0.569,-1.557) -- (-0.611,-1.511) -- (-0.650,-1.464) -- (-0.686,-1.417) -- (-0.720,-1.369) -- (-0.751,-1.321) -- (-0.780,-1.273) -- (-0.806,-1.225) -- (-0.830,-1.177) -- (-0.852,-1.129) -- (-0.872,-1.082) -- (-0.889,-1.034) -- (-0.905,-0.987) -- (-0.919,-0.941) -- (-0.930,-0.895) -- (-0.941,-0.849) -- (-0.949,-0.804) -- (-0.956,-0.760) -- (-0.962,-0.716) -- (-0.966,-0.673) -- (-0.969,-0.630) -- (-0.970,-0.588) -- (-0.971,-0.547) -- (-0.970,-0.506) -- (-0.968,-0.466) -- (-0.965,-0.427);
\draw[thin,gray] (-0.371,-1.729) -- (-0.310,-1.700) -- (-0.252,-1.670) -- (-0.197,-1.638) -- (-0.145,-1.605) -- (-0.095,-1.572) -- (-0.048,-1.537) -- (-0.004,-1.502) -- (0.038,-1.467) -- (0.078,-1.431) -- (0.116,-1.395) -- (0.151,-1.359) -- (0.185,-1.322) -- (0.216,-1.286) -- (0.246,-1.249) -- (0.274,-1.213) -- (0.300,-1.176) -- (0.325,-1.140) -- (0.348,-1.104) -- (0.370,-1.068) -- (0.391,-1.033) -- (0.410,-0.998) -- (0.428,-0.963) -- (0.445,-0.928) -- (0.461,-0.893) -- (0.475,-0.859) -- (0.489,-0.825) -- (0.502,-0.791) -- (0.513,-0.758) -- (0.524,-0.725);
\draw[thick,dashed] (1.689,0.551) -- (1.710,0.677) -- (1.720,0.810) -- (1.719,0.949) -- (1.703,1.094) -- (1.673,1.243) -- (1.625,1.395) -- (1.560,1.547) -- (1.475,1.695) -- (1.370,1.838) -- (1.246,1.971) -- (1.103,2.091) -- (0.944,2.195) -- (0.771,2.279) -- (0.587,2.341) -- (0.396,2.379) -- (0.203,2.393) -- (0.011,2.383) -- (-0.175,2.349) -- (-0.352,2.294) -- (-0.517,2.220) -- (-0.667,2.131) -- (-0.802,2.028) -- (-0.920,1.916) -- (-1.022,1.796) -- (-1.108,1.673) -- (-1.179,1.547) -- (-1.235,1.421) -- (-1.277,1.297) -- (-1.308,1.175);
\draw[thick,dashed] (-1.308,1.175) -- (-1.425,1.131) -- (-1.544,1.075) -- (-1.663,1.005) -- (-1.779,0.921) -- (-1.892,0.822) -- (-1.999,0.708) -- (-2.097,0.578) -- (-2.183,0.433) -- (-2.255,0.274) -- (-2.310,0.103) -- (-2.344,-0.078) -- (-2.357,-0.265) -- (-2.346,-0.455) -- (-2.311,-0.645) -- (-2.253,-0.828) -- (-2.172,-1.003) -- (-2.070,-1.165) -- (-1.951,-1.310) -- (-1.818,-1.438) -- (-1.673,-1.546) -- (-1.522,-1.635) -- (-1.367,-1.703) -- (-1.211,-1.752) -- (-1.056,-1.783) -- (-0.906,-1.798) -- (-0.761,-1.798) -- (-0.623,-1.786) -- (-0.493,-1.762) -- (-0.371,-1.729);
\draw[thick,dashed] (-0.371,-1.729) -- (-0.274,-1.810) -- (-0.166,-1.887) -- (-0.046,-1.957) -- (0.086,-2.018) -- (0.229,-2.068) -- (0.384,-2.105) -- (0.547,-2.127) -- (0.718,-2.131) -- (0.895,-2.115) -- (1.074,-2.077) -- (1.251,-2.016) -- (1.423,-1.932) -- (1.585,-1.826) -- (1.734,-1.699) -- (1.867,-1.554) -- (1.979,-1.393) -- (2.069,-1.221) -- (2.136,-1.042) -- (2.180,-0.859) -- (2.200,-0.677) -- (2.199,-0.499) -- (2.179,-0.328) -- (2.141,-0.166) -- (2.089,-0.016) -- (2.024,0.123) -- (1.950,0.249) -- (1.868,0.362) -- (1.780,0.462) -- (1.689,0.551);
\fill (-0.965,-0.427) circle (1.3pt); \node at (-1.226,-0.529) {$A$};
\fill (0.524,-0.725) circle (1.3pt); \node at (0.711,-0.934) {$B$};
\fill (0.071,0.727) circle (1.3pt); \node at (0.116,1.003) {$C$};
\fill (1.689,0.551) circle (1.3pt); \node at (1.953,0.645) {$X$};
\fill (-1.308,1.175) circle (1.3pt); \node at (-1.506,1.373) {$Y$};
\fill (-0.371,-1.729) circle (1.3pt); \node at (-0.422,-2.004) {$Z$};
\end{tikzpicture}
\caption{Stereographic picture of Example~\ref{ex:DNV}: the triangle $ABC$ (thick), the isosceles triangles $XBC$, $YCA$, $ZAB$ with apex angle $2\pi/3$ (thin), and the equilateral triangle $XYZ$ with side $\arccos(-1/3)\approx109.47^\circ$ (dashed).}
\label{fig:dnv}
\end{figure}
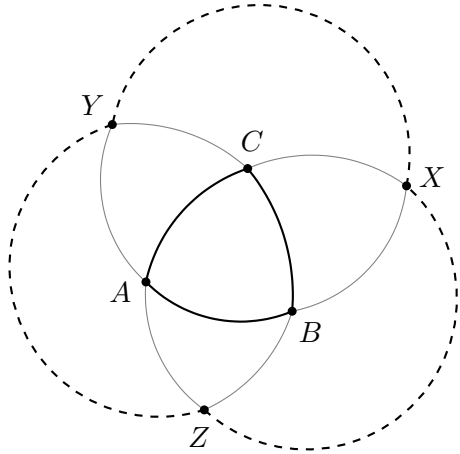

\begin{example}[both constructions]\label{ex:both}
Let $\cos^2\beta=11/60$, so that $\beta\approx64.65^\circ$, and take $\cos a=\cos b=-3/5$ and $\cos c=1/5$. The sides of $ABC$ are approximately $126.87^\circ$, $126.87^\circ$, $78.46^\circ$. We have $(x,y,z)=(1,1,5)/\sqrt{11}$, $e_2=1$, $e_1-e_3=72/(11\sqrt{11})>0$ and $\alpha=0$. Hence $\Theta=\Omega=180^\circ$, and the parameter point lies on both quadrics. The outward triangle of centres has $\cos\ell=-\cos\beta/(1+\cos\beta)$, i.e.~$\ell\approx107.45^\circ$, while the inward one has $\cos\ell=\cos\beta/(1-\cos\beta)$, i.e.~$\ell\approx41.51^\circ$. The sign condition of Theorem~\ref{thm:quadric} selects $\cQ_{\cos\beta}$ for the outward construction and $\cQ_{-\cos\beta}$ for the inward one. Accordingly, the outward $DEF$ is equilateral, with $\cos m=\cos\beta/(1-\cos\beta)$, whereas the inward $DEF$ is not: its pairwise inner products are approximately $-0.6468$, $-0.2391$, $-0.2391$.
\end{example}

For $\cos\beta=1/3$, the two quadrics are
\[
\cQ_{1/3}: 3(x^2+y^2+z^2)+2e_2=20, \ \ \cQ_{-1/3}: 3(x^2+y^2+z^2)-2e_2=16,
\]
the cube is $(0,2\sqrt2)^3$, and the possible sides satisfy $\cos\ell=-1/4$ or $\cos\ell=1/2$.

\begin{example}[$\cos\beta=1/3$, second outward family]\label{ex:second}
The point $(2,\frac83,\frac{10}9)$ lies on $\cQ_{-1/3}$, and $e_1-e_3=-4/27<0$. Its side cosines are $1/9$, $65/81$, $-367/729$, so that the sides are approximately $83.62^\circ$, $36.63^\circ$, $120.23^\circ$. Here $\Theta\approx89.11^\circ$ and $\Omega\approx1.78^\circ$, and the exact relation is $\Theta+\Omega/2=90^\circ$. Thus the outward $XYZ$ is equilateral with side $60^\circ$. The triangle $DEF$ has sides approximately $54.53^\circ$, $127.22^\circ$, $167.24^\circ$, so it is not equilateral. For comparison, the point $(1,2,\frac{2\sqrt3-3}{3})\in\cQ_{1/3}$ satisfies $e_1-e_3>0$ and gives an outward equilateral $XYZ$ with $\cos\ell=-1/4$. These two inputs realize the two outward sizes for the same $\beta$.
\end{example}

\begin{example}[$\cos\beta=1/5$, second inward family]\label{ex:fourth}
We now take $\cos\beta=1/5<1/4$. The point $(\frac43,2,\frac{10}3)$ lies on $\cQ_{1/5}$, whose equation is $5(x^2+y^2+z^2)+2e_2=112$, and $e_1-e_3=-20/9<0$. The side cosines are $-7/9$, $-3/5$, $-7/225$, so that the sides are approximately $141.06^\circ$, $126.87^\circ$, $91.78^\circ$. The area is $\Omega\approx340.27^\circ$, while the base angles $\operatorname{arccot}(4/3)$, $\operatorname{arccot}2$, $\operatorname{arccot}(10/3)$ have sum $\Theta\approx80.13^\circ$. Thus $\Theta-\Omega/2=-90^\circ$, and the inward $XYZ$ is equilateral with $\cos\ell=-1/6$. Proposition~\ref{prop:DEF} shows that $DEF$ is equilateral as well, with $\cos m=8/9$.
\end{example}

\begin{remark}[Small triangles]\label{rem:limit}
We keep $\beta$ fixed and consider triangles with $a=\rho a_0+o(\rho)$, $b=\rho b_0+o(\rho)$, $c=\rho c_0+o(\rho)$ as $\rho\to0$, where $a_0,b_0,c_0$ are the sides of a nondegenerate Euclidean triangle. We put $\diam=\max(a,b,c)$. Then $\Omega\to0$ and each base angle tends to $\pi/2-\beta$. For either construction, $\cos\Phi\to\cos(3\pi/2-3\beta)=-\sin3\beta$, which is nonzero unless $\beta=\pi/3$. For $\beta=\pi/3$, the next term explains the relation with the planar theorem. From $\sin\theta_a=1/(2\cos(a/2))$ we obtain $\theta_a=\pi/6+a^2/(8\sqrt3)+O(a^4)$. Let $\Delta$ be the area of the Euclidean triangle with sides $a,b,c$, and put $s_0=(a+b+c)/2$. L'Huilier's formula
\[
\tan\frac\Omega4=\sqrt{\tan\frac{s_0}{2}\tan\frac{s_0-a}{2}\tan\frac{s_0-b}{2}\tan\frac{s_0-c}{2}}
\]
and Heron's formula give $\Omega=\Delta+O(\diam^4)$. Therefore, for the inward construction,
\[
\Theta-\frac\Omega2-\frac\pi2=\frac{1}{8\sqrt3}\bigl(a^2+b^2+c^2-4\sqrt3\,\Delta\bigr)+O(\diam^4).
\]
The quadratic term is nonnegative by Weitzenb\"ock's inequality \cite{BDJ+69}, and vanishes if and only if $a=b=c$. In other words, to leading order it measures how far the inward construction is from the equation $\Theta-\Omega/2=\pi/2$. For the outward construction, $\Theta+\Omega/2$ tends to $\pi/2$, whereas its only admissible value is $3\pi/2$. We stress that the expansion only describes this limiting behaviour; the exact exclusion of non-equilateral inward solutions is Corollary~\ref{cor:small}.
\end{remark}

\section{Further questions}\label{sec:questions}

\begin{question}
The appendix proves the angle criterion by composing the rotations about $X,Y,Z$, and uses quaternions to compute the composition. Can this argument be replaced by a synthetic proof?
\end{question}

\begin{question}
What is the corresponding classification in the hyperbolic plane? The Euclidean chord-angle definition of a $\beta$-point has no direct intrinsic analogue, but Remark~\ref{rem:equivalent} suggests erecting isosceles hyperbolic triangles with apex angle $2\beta$. Their base angles are automatically smaller than $\pi/2-\beta$, so that there is no ambiguity between acute and obtuse base angles. For $\beta=\pi/3$, this construction should be compared with \cite{DNV25}.
\end{question}

\begin{question}
For which $ABC$ and $\beta$ is $DEF$ equilateral? Proposition~\ref{prop:DEF} gives a sufficient condition, but Remark~\ref{rem:DEF}(2) shows that the answer extends beyond inputs with equilateral $XYZ$.
\end{question}

\begin{question}
The Petr--Douglas--Neumann theorem extends Napoleon's theorem to polygons (see the references in \cite{DNV25}). Is there a spherical counterpart obtained by imposing similarity on chord polygons?
\end{question}

\appendix
\section{A rotation interpretation}\label{sec:appendix}

We keep the notation and assumptions of \S\ref{sec:construction}, and recall that $\Phi=\Theta+\eps\Omega/2$. We write a quaternion as a pair $(u,v)$ with $u\in\R$ and $v\in\R^3$, with the product
\[
(u,v)(u',v')=(uu'-\ip{v}{v'},\,uv'+u'v+v\times v').
\]
We identify $U\in\R^3$ with the pure quaternion $(0,U)$, and write $\bar q$ for the conjugate of $q$. Recall that the unit quaternion $r_N(\lambda)\defeq(\cos\lambda,\sin\lambda\,N)$, where $|N|=1$, acts by $U\mapsto r_N(\lambda)U\overline{r_N(\lambda)}$, and represents the rotation through the angle $2\lambda$ about the oriented axis $N$ (right-hand rule); products of quaternions correspond to compositions of rotations. For distinct, non-antipodal unit vectors $U,V$, we set
\[
h_{UV}\defeq\frac{(1+\ip{U}{V},U\times V)}{\sqrt{2(1+\ip{U}{V})}}.
\]
This is the quaternion of the rotation through $d(U,V)$ about $U\times V$, which carries $U$ to $V$. More generally, if a unit quaternion $q$ carries $U$ to $V$, then $q\,r_U(\lambda)\bar q=r_V(\lambda)$, i.e.~$r_V(\lambda)q=q\,r_U(\lambda)$.

Let $R_X,R_Y,R_Z$ be the rotations about the axes $X,Y,Z$ through the signed angle $-2\eps\beta$. They are represented by
\[
q_X=(\cos\beta,-\eps\sin\beta\,X),\ q_Y=(\cos\beta,-\eps\sin\beta\,Y),\ q_Z=(\cos\beta,-\eps\sin\beta\,Z).
\]
The factorization of Lemma~\ref{lem:factor} below shows that $R_X$ first rotates about $B$ and then carries $B$ to $C$ along the shorter arc $BC$. In particular $R_X(B)=C$, and cyclically $R_Y(C)=A$ and $R_Z(A)=B$.

\begin{lemma}\label{lem:factor}
We have $q_X=h_{BC}\,r_B\bigl(-\eps(\pi/2-\theta_a)\bigr)$, and cyclically.
\end{lemma}

\begin{proof}
We write $h_{BC}=(\cos(a/2),\sin(a/2)\,n)$, where $n=(B\times C)/\sin a$, and put $\mu=-\eps(\pi/2-\theta_a)$, so that $\cos\mu=\sin\theta_a$ and $\sin\mu=-\eps\cos\theta_a$. As $n\perp B$, the real part of $h_{BC}r_B(\mu)$ is $\cos(a/2)\sin\theta_a=\cos\beta$. For its vector part, we use $n\times B=(C-pB)/\sin a$ and $2\cos^2(a/2)=1+p$ to obtain
\[
\cos\tfrac a2\sin\mu\,B+\cos\mu\sin\tfrac a2\,n+\sin\tfrac a2\sin\mu\,n\times B
=\frac{-\eps\cos\theta_a\,(B+C)+\sin\theta_a\,B\times C}{2\cos\frac a2}.
\]
By \eqref{eq:theta}, we have $\sin\theta_a=\cos\beta/\cos\frac a2$ and $\cos\theta_a=x\sin\theta_a$. Hence the vector part equals $\cos\beta\,(-\eps x(B+C)+B\times C)/(1+p)=-\eps\sin\beta\,X$ by \eqref{eq:X}.
\end{proof}

\begin{proposition}\label{prop:rotation}
We have $q_Zq_Yq_X=(-\sin\Phi,\eps\cos\Phi\,B)$. In particular, $R_Z\circ R_Y\circ R_X$ is the rotation about $B$ through the angle $\eps(2\Phi-\pi)$ modulo $2\pi$, and it is the identity if and only if $\cos\Phi=0$.
\end{proposition}

\begin{proof}
We write $\mu_a=-\eps(\pi/2-\theta_a)$, and similarly $\mu_b,\mu_c$. By Lemma~\ref{lem:factor} and the rule $r_V(\lambda)q=q\,r_U(\lambda)$, applied to $h_{CA}$ (which carries $C$ to $A$) and to $h_{BC}$ (which carries $B$ to $C$), we have
\[
q_Zq_Yq_X=h_{AB}r_A(\mu_c)h_{CA}r_C(\mu_b)h_{BC}r_B(\mu_a)
=H\,r_B(\mu_a+\mu_b+\mu_c), \text{ where } H=h_{AB}h_{CA}h_{BC}.
\]
We now compute $H$. Using $\ip{C\times A}{B\times C}=pq-r$ and $(C\times A)\times(B\times C)=-\delta C$, we find
\[
h_{CA}h_{BC}=\frac{\bigl(\alpha,V\bigr)}{4\cos\frac a2\cos\frac b2},\ \ V=(1+q)\,B\times C+(1+p)\,C\times A-\delta C .
\]
Multiplying by $h_{AB}=(1+r,A\times B)/(2\cos\frac c2)$, and using $\ip{A\times B}{B\times C}=rp-q$, $\ip{A\times B}{C\times A}=qr-p$, $\ip{A\times B}{C}=\delta$ and \eqref{eq:gram}, we see that the real part of $8g\,H$ is
\[
(1+r)\alpha-(1+q)(rp-q)-(1+p)(qr-p)+\delta^2=2\alpha .
\]
Applied from right to left, the three factors of $H$ take $B$ to $C$, then to $A$, and finally back to $B$. Thus $H$ fixes $B$, and its vector part $H_{\mathrm{vec}}$ is a multiple of $B$. The identities
\[
\ip{V}{B}=\delta,\ \ \ip{(A\times B)\times V}{B}=\ip{A\times B}{V\times B}=\delta(1-r)
\]
give $\ip{8g\,H_{\mathrm{vec}}}{B}=(1+r)\delta+\delta(1-r)=2\delta$. We conclude that $H=(\alpha,\delta B)/(4g)$, which equals $r_B(\Omega/2)$ by Lemma~\ref{lem:area}. Since $\mu_a+\mu_b+\mu_c=-\eps(3\pi/2-\Theta)$ and $\Omega/2+\eps\Theta=\eps\Phi$, we deduce that
\[
q_Zq_Yq_X=r_B\Bigl(\frac\Omega2-\eps\frac{3\pi}{2}+\eps\Theta\Bigr)=r_B\Bigl(\eps\bigl(\Phi-\tfrac{3\pi}{2}\bigr)\Bigr)=(-\sin\Phi,\eps\cos\Phi\,B).
\]
The rotation angle is $2\eps(\Phi-3\pi/2)\equiv\eps(2\Phi-\pi)$ modulo $2\pi$, and it vanishes modulo $2\pi$ if and only if $\cos\Phi=0$.
\end{proof}

We now deduce the criterion of Theorem~\ref{thm:angle} directly from this rotation formula. Write $q_Zq_Yq_X=(u,wB)$, where $u=-\sin\Phi$ and $w=\eps\cos\Phi$. Conjugating first by $q_X$ and then by $q_Y$ gives $q_Xq_Zq_Y=(u,wC)$ and $q_Yq_Xq_Z=(u,wA)$. We take real parts in the identities $q_Yq_X=q_Z^{-1}(q_Zq_Yq_X)$, $q_Zq_Y=(q_Zq_Yq_X)q_X^{-1}$ and $q_Xq_Z=(q_Xq_Zq_Y)q_Y^{-1}$. Since $\ip{Z}{B}=z\cot\beta$, $\ip{X}{B}=x\cot\beta$ and $\ip{Y}{C}=y\cot\beta$, this gives
\begin{align*}
\cos^2\beta-\sin^2\beta\,\ip{X}{Y}&=u\cos\beta-\eps w\sin\beta\cot\beta\,z,\\
\cos^2\beta-\sin^2\beta\,\ip{Y}{Z}&=u\cos\beta-\eps w\sin\beta\cot\beta\,x,\\
\cos^2\beta-\sin^2\beta\,\ip{Z}{X}&=u\cos\beta-\eps w\sin\beta\cot\beta\,y.
\end{align*}
For a non-equilateral $ABC$, the coordinates $x,y,z$ are not all equal, so that the three inner products are equal if and only if $w=0$, i.e.~$\cos\Phi=0$. In this case $u=-\sigma$, and their common value is $\rho=(\cos^2\beta+\sigma\cos\beta)/\sin^2\beta=\sigma\cos\beta/(1-\sigma\cos\beta)$. If $\sigma=1$, the bound $\rho\le1$ gives $\cos\beta\le1/2$, and equality would force $X=Y=Z$, hence $x=y=z$ by Lemma~\ref{lem:construction}, which is excluded; thus $0<\rho<1$. If $\sigma=-1$, then $-1/2<\rho<0$. In either case the Gram eigenvalues $1-\rho,1-\rho,1+2\rho$ are positive, which proves nondegeneracy and completes the second proof.

\section*{Acknowledgments}
The author thanks Claude (Anthropic) and GPT Sol 6.1 (OpenAI) for their assistance during the preparation of this paper and checking calculations.

\section*{Declarations}
\textbf{Conflict of interest.} The author declares that there is no conflict of interest.\\
\textbf{Data availability.} No external datasets were used. The examples can be reproduced from the formulas in the paper.

\begin{flushright}
	Quang Hung Tran\\
	\small High School for Gifted Students, Vietnam National University, Hanoi, Vietnam\\
	\small Email: tranquanghung@hus.edu.vn
\end{flushright}

\end{document}